\documentclass{amsart}

\usepackage{amsmath,amstext,amsthm,amscd,amsopn,verbatim,amssymb}
\usepackage{hyperref}
\usepackage{graphicx}
\usepackage{color}
\usepackage{multirow}
\usepackage{array}
\usepackage[all,cmtip]{xy}

\newtheorem{thm}{Theorem}
\newtheorem{prop}[thm]{Proposition}
\newtheorem{lemma}[thm]{Lemma}
\theoremstyle{definition}
\newtheorem{defi}[thm]{Definition}

\theoremstyle{remark}

\newtheorem{rem}[thm]{Remark}

\newcommand{\alg}[1]{\mathfrak{{#1}}}
\newcommand{\skp}[2]{\left\langle {#1}, {#2} \right\rangle} 

\newcommand{\co}[2]{\left[{#1},{#2}\right]} 

\newcommand{\eref}[1]{(\ref{#1})} 

\newcommand{\morphU}{{\mathcal{U} }}
\newcommand{\mU}{{\morphU }}
\newcommand{\mV}{{\mathcal{V} }}

\newcommand{\p}{\partial}

\newcommand{\End}{\mathop{End}}

\newcommand{\C}{{\mathbb{C}}}

\newcommand{\Z}{{\mathbb{Z}}}

\DeclareMathOperator{\dv}{div}

\newcommand{\ndash}{\nobreakdash-\hspace{0pt}}

\begin{document}
\title{On $L_\infty$-morphisms of cyclic chains}

\author{Alberto S. Cattaneo}
\author{Giovanni Felder}
\author{Thomas Willwacher}


\thanks{This work been partially supported by SNF Grants
200020-121640/1 and 200020-105450, by the European Union through the FP6 Marie
Curie RTN ENIGMA (contract number MRTN-CT-2004-5652), and by the
European Science Foundation through the MISGAM program.}
\date{}

\begin{abstract}
Recently the first two authors \cite{cafe} constructed
an $L_\infty$-morphism using the $S^1$-equivariant version
of the Poisson Sigma Model (PSM). Its role in deformation
quantization was not entirely clear. We give here a
``good'' interpretation and show that the resulting
formality statement is equivalent to formality on
cyclic chains as conjectured by Tsygan and proved
recently by several authors \cite{dtt}, \cite{me}.

\end{abstract}
\maketitle

\section{Introduction and Structure}
We begin by drawing the big picture; precise
definitions will be given below.
\subsection{Big picture on cochains}
Let $M$ be a smooth $d$-dimensional manifold
and $A=C^\infty(M)$ ($A_c=C^\infty_c(M)$) the
commutative algebras of smooth (compactly
supported) functions. We denote by $T^\bullet$
the dgla of multivector fields and by $C^\bullet(A)$
the multidifferential Hochschild complex.
Kontsevich's famous Formality Theorem asserts
that there is an $L_\infty$-quasi\ndash isomorphism
of dglas
\[
\mU_K \colon T^\bullet \to C^\bullet.
\]
Next, assume that $M$ is orientable\footnote{This is
actually not necessary, but we will assume it for
simplicity.} and pick a volume form $\Omega$. This
endows $T^\bullet$ with an additional differential
$\dv_\Omega$, the divergence, that is compatible
with the Schouten bracket on $T^\bullet$. We will
denote the dgla
$(T^\bullet[[u]],u\dv_\Omega,\co{\cdot}{\cdot}_S)$
shortly by $T^\bullet[[u]]$. Here $u$ is a formal
parameter of degree +2. There is a morphism of dglas
\[
 T^\bullet[[u]]\stackrel{u=0}{\longrightarrow}T^\bullet.
\]
We denote the composition of this morphism with
$\mU_K$ also by $\mU_K$ for simplicity.

\subsection{Big picture on chains}
Let us turn to homology. Denote the negatively graded
Hochschild (chain) complex by $C_\bullet(A)=C_\bullet(A,A)$.
It is a mixed complex, with the Hochschild differential
$b$ of degree +1 and with the Rinehart (or Connes)
differential $B$ of degree -1. The cohomology $H_\bullet(A)$
of $C_\bullet(A)$ wrt. the differential $b$ is the de
Rham complex $(\Omega^{-\bullet}(M),{\rm d})$, which we
view as a bicomplex with vanishing first differential.

$C_\bullet(A)$ also carries a compatible dgla module
structure over the Hochschild cochains $C^\bullet(A)$.
Pulling back this module structure along $\mU_K$, we
obtain an $L_\infty$-module structure over multivector
fields $T^\bullet$. The Hochschild Formality Theorem
on chains \cite{shoikhetchain,dolgushev,tsygan} states
that there is a quasi\ndash isomorphism of $L_\infty$-modules
over $T^\bullet$
\[
\mV: C_\bullet(A) \to \Omega^{-\bullet}(M)
\]
Actually, this morphism is compatible with the additional
second differentials $B$ and ${\rm d}$ on both sides.
Hence we obtain an $L_\infty$-quasi\ndash isomorphism
\[
\mV: (C_\bullet(A)[[u]],b+uB) \to (\Omega^{-\bullet}(M)[[u]], u{\rm d})
\]
This last statement is known as the Cyclic Formality
Theorem on chains \cite{me,dtt,tsygan}.

\subsection{Dual picture}
Recall that $A=C^\infty(M)$. The following statement
is a particularly simple case of van den Bergh duality
(note the negative grading on the left)
\[
 H_\bullet(A,A) \cong H^{d+\bullet}(A, \Omega^d(M))
\]
Concretely, the left hand side is $\Omega^{-\bullet}(M)$,
and the right hand side is
$VT^{d+\bullet}:=T^{d+\bullet}\otimes \Omega^d(M)$.
The isomorphism from right to left is by contraction.
Note that we can pull back the de Rham differential
along this isomorphism, obtaining a differential ``$\dv$''
on $VT^\bullet$. Note in particular that this differential
$\dv$ does not depend on a choice of volume, in contrast
to the $\dv_\Omega$ defined before.

The dualized Hochschild formality theorem on chains
states that there is a quasi\ndash isomorphism of $L_\infty$-modules
\[
\mV^* : VT^\bullet \to C^{\bullet}(A,\Omega^d)\, .
\]
The dualized cyclic formality theorem states that this
morphism is compatible with the additional differentials
$\dv$ on the left and the (adjoint of the) Connes differential $B$ on the right.

We will only consider such morphisms that are differential
operators in each argument. In this case there is a canonical
way to obtain an adjoint morphism $\mV^*$ from the ``direct''
one $\mV$ and vice versa. Concretely, there is a pairing
between $C^\bullet(A,\Omega^d)$ and $C_\bullet(A_c)$ given by
\[
 \skp{\phi}{a_0\otimes\cdots \otimes a_n} = \int_M a_0 \phi(a_1,\dots,a_n)
\]
and a pairing between $VT^\bullet(M)$ and $\Omega^\bullet(M)$ given by
\[
 \skp{\gamma \Omega}{\alpha} = \int_M (\iota_\gamma \alpha) \Omega \, .
\]
Here the insertion $\iota_\gamma$ is defined such that $\iota_{\gamma_1\wedge \gamma_2}=\iota_{\gamma_1}\iota_{\gamma_2}$.
One can see that to any direct multidifferential $L_\infty$
morphism $\mV$ there is a unique morphism $\mV^*$ such that
\[
\skp{\gamma \Omega}{\mV(a_0\otimes\cdots \otimes a_n)}=\pm \skp{\mV^*(\gamma \Omega)}{a_0\otimes\cdots \otimes a_n}\, .
\]
It follows that the direct and adjoint (multidifferential)
formality statements are equivalent.

\begin{rem}[on quantization]
The cohomology $H^0(A,\Omega^d)$ is important because it
classifies smooth traces on $A_c$, i.e., top degree
differential forms $\Omega$ such that the functional
$f\mapsto \int_Mf\Omega$ is a trace on $A_c$. Of course,
in the current commutative setting, these are just all
top degree differential forms. However, due to dual
Hochschild formality we can quantize. Let $A_\star$
be the algebra $C^\infty(M)[[\hbar]]$ with
the Kontsevich star product \cite{kontsevich} associated to a Poisson structure $\pi$. The relevant
cohomology is then
$H^0(A_\star,\Omega^d_\star)\cong \{\omega\in\Omega^d(M)[[\hbar]]\mid \dv_\omega \pi=0 \}$.
The quantized bimodule structure on $\Omega^d_\star=\Omega^d[[\hbar]]$
is defined such that for all $a,b\in A_c$, $\omega \in \Omega^d_\star$
\begin{align*}
 \int_M a \cdot (L_b\omega) = \int_M (a\star b) \cdot \omega = \int_M b \cdot (R_a \omega)\, .
\end{align*}
\end{rem}

\subsection{Other module structures}
The cyclic chain formality morphisms above are
quasi\ndash isomorphisms of $L_\infty$-modules over
$(T^\bullet(M),0,\co{\cdot}{\cdot}_S)$. One may
be tempted to replace this latter dgla by its
``cyclic'' counterpart
$(T^\bullet(M)[[u]],u\dv_\Omega,\co{\cdot}{\cdot}_S)$,
and ask whether the above formality statements remain
true. Of course, if we use the module structures obtained
via pulling back along the dgla morphism
\[
T^\bullet[[u]] \stackrel{u=0}{\to} T^\bullet
\]
the new formality statements will be equivalent to
the original ones. However, one may try to change
the module structures. We will only consider changing
the module structure on the classical (differential forms)
side.\footnote{One can also ``naturally'' change
the action on the Hochschild side. but we don't
discuss it here.} We show in section \ref{sec:diffforms} that
there is a whole family of dgla actions $L^{(t)}$
reducing to the original Lie derivative action for $t=0$.
However, all these module structures will be shown
to be $L_\infty$-quasi-isomorphic in Proposition \ref{prop:Ltqiso}.

\subsection{Meaning of the PSM morphism}
Using the $S^1$-equivariant version of the Poisson
Sigma Model the first two authors \cite{cafe} recently
constructed an $L_\infty$-morphism $\mV_{PSM,orig}$,
the ``PSM morphism''. This paper is devoted to clarifying
the meaning of this morphism, which was not entirely
clear. To do this, we will reinterpret $\mV_{PSM,orig}$
slightly, yielding a morphism $\mV_{PSM}^*$.
Concretely, we introduce a new complex $E^\bullet$
which is quasi-isomorphic (as bicomplex and
$C^\bullet(A)$-module) to $C^\bullet(A, \Omega^d)$.
The morphism $\mV_{PSM}^*$ can then be understood as
an adjoint cyclic chain formality morphism on
\[
\mV^*_{PSM}: T^\bullet[[u]]\cong VT^\bullet[[u]] \to E^\bullet[[u]]\, .
\]
Here the action of $T^\bullet[[u]]$ on the very left
is the adjoint action, on the middle it is the (dual
of the) action $L^{(1)}$, and on the
right it is the action defined through pullback via
$\mU_K$. The isomorphism on the left is defined by
choosing a volume form.

\subsection{Organisation of the paper}
The remainder of the paper is divided into two parts:
\begin{enumerate}
\item In the first part we introduce the structures
involved, i.e., the Hochschild and cyclic chain and
cochain complexes. Here there are two novel aspects:
(i) We introduce the natural ``extended'' complex
$E^\bullet$ mentioned above that allows us to give
a nice interpretation of the PSM morphism and (ii)
we introduce the aforementioned family $L^{(t)}$ of
$T^\bullet[[u]]$-actions on differential forms that
was (to our knowledge) not studied before.

\item In the second part we define $\mV_{PSM}^*$ and
prove the formality statement made above.
\end{enumerate}

\section{Part I: The objects of study}
In this section we define the different complexes that
will be related to each other through formality morphisms.
 Each complex can either constitute a differential
graded Lie algebra (dgla) or serve as a module over
one of the dglas. We will indicate the roles in the
titles of each subsection. Of course, every dgla is
also a module over itself.

\subsection{Multivector fields $T^\bullet$ (dgla)}
The algebra of multivector fields on $M$, $T^\bullet(M)$,
is the algebra of smooth sections of $\wedge^\bullet TM$.
There is a Lie bracket $\co{\cdot}{\cdot}_{S}$ on $T^{\bullet+1}(M)$,
the Schouten bracket, extending the Lie derivative and
making $T^\bullet(M)$ a Gerstenhaber algebra. More concretely,
\begin{multline*}
\co{v_1\wedge\cdots \wedge v_m}{w_1\wedge\cdots \wedge w_n}_{S}
=\\=
\sum_{i=0}^n\sum_{j=0}^n (-1)^{i+j}\co{v_i}{w_j}\wedge v_1\wedge \cdots \wedge\hat{v}_i\wedge\cdots \wedge v_m\wedge w_1\wedge\cdots \wedge \hat{w}_j\wedge\cdots \wedge w_n\, .
\end{multline*}

Assume now that $M$ is oriented, with volume form $\Omega$.
Contraction with $\Omega$ defines an isomorphism
$T^\bullet(M)\to \Omega^{d-\bullet}(M)$. The
divergence operator $\dv_\Omega$ on $T^\bullet(M)$
is defined as the pull-back of the de Rham
differential ${\rm d}$ on $\Omega^{\bullet}(M)$
under this isomorphism. Concretely
\[
 \iota_{\dv_\Omega \gamma}\Omega = {\rm d}\iota_\gamma \Omega\, .
\]
One can check that $\dv_\Omega$
is a derivation with respect to the Schouten bracket,
i.e.,\footnote{Actually $\dv_\Omega$ is a BV operator
generating $\co{\cdot}{\cdot}_S$ for any volume form $\Omega$.}
\[
\dv_\Omega \co{\gamma_1}{\gamma_2}_{S} = \co{\dv_\Omega\gamma_1}{\gamma_2}_{S}+ (-1)^{k_1-1}\co{\gamma_1}{\dv_\Omega \gamma_2}_{S}\, .
\]
Introducing a new formal variable $u$ of degree +2,
the complex $T^{\bullet+1}(M)[[u]]$ is a dgla with
differential $u\dv_\Omega$ and bracket the $u$-linear
extension of the Schouten bracket.

Hence we have two dglas, $T^{\bullet+1}(M)$ and
$T^{\bullet+1}(M)[[u]]$, related by a dgla morphism
\[
T^{\bullet+1}(M)[[u]] \stackrel{u=0}{\longrightarrow} T^{\bullet+1}(M).
\]
This morphism in particular allows us to view
any $T^{\bullet+1}(M)$-module also as $T^{\bullet+1}(M)[[u]]$-module.

\subsection{Hochschild cochains $C^\bullet(A)$ (dgla)}
The normalized multidifferential Hochschild complex
$C^\bullet(A)$ is the complex of $\bullet$-differential
operators, which vanish upon insertion of a constant
function in any of its arguments. E.g., $C^1(M)$ are
differential operators $D$ such that $D1=0$. $C^{\bullet+1}(A)$
is a differential graded Lie algebra with the Gerstenhaber bracket
\begin{align*}
[\phi,\psi]_G(a_1,\dots , a_{p+q-1})
&= \phi(\psi(a_1,\dots,a_q),a_{q+1},\dots, a_{p+q-1}) \\
&\quad\quad +(-1)^{q-1}\phi(a_1,\psi(a_2,\dots,a_{q+1}),a_{q+2},\dots, a_{p+q-1}) \\
&\quad\quad \pm \dots \\
&\quad\quad +(-1)^{(p-1)(q-1)} \phi(a_1,\dots,\psi(a_p,\dots,a_{p+q-1})) \\
&\quad\quad - (-1)^{(p-1)(q-1)} (\phi \leftrightarrow \psi)
\end{align*}
for $\phi \in C^p(A), \psi \in C^q(A)$, and the Hochschild differential
\[
b^H = \co{m_0}{\cdot}_G.
\]
Here $m_0\in C^2(A)$ is the usual (commutative) multiplication of functions.

\subsection{The differential forms $\Omega^\bullet(M)$ (module)}
\label{sec:diffforms}
Let $\Omega^\bullet=\Omega^\bullet(M)$ be the graded
algebra of differential forms on $M$, with negative
grading. Let ${\rm d}={\rm d}_{dR}$ be the de Rham
differential. Denote the insertion operators by
$\iota_\gamma$. They take a form and contract it
with the multivector field $\gamma$. The signs
are such that
\begin{align*}
 \iota: T^\bullet &\to \End(\Omega^\bullet) \\
\gamma &\mapsto \iota_\gamma
\end{align*}
is a morphism of graded algebras. For example, for
a function $f$, $\iota_f$ is multiplication by $f$,
for a vector field $\xi$, $\iota_\xi$ is a derivation
of the dga $\Omega^\bullet$ and for any multivector
fields $\gamma$, $\nu$,
$\iota_{\gamma\wedge \nu}=\iota_\gamma \iota_\nu$.
The Lie derivative $L$ is:
\[
L_\gamma := \co{{\rm d}}{\iota_\gamma} \, .
\]
It satisfies the following relation, which can
alternatively be taken as the definition of the
Schouten bracket.
\[
 \iota_{\co{\gamma}{\nu}_S} = \co{\iota_\gamma}{L_\nu} = (-1)^{|\gamma|}\co{L_\gamma}{\iota_\nu}
\]
It follows that $L$ forms a representation of the differential
graded Lie algebra $T^{\bullet+1}$.
Here and everywhere in the paper the degrees $|\gamma|$
are such that $\gamma\in T^{|\gamma|+1}$.

Next consider module structures on $(\Omega^\bullet[[u]], u {\rm d})$
over the dgla $(T^\bullet[[u]],\co{\cdot}{\cdot}_{S},u \dv_\Omega)$.
Let us introduce a family of actions $L_\gamma^{(t)}$
as follows. Let $S^{(t)}$ be the $u$-scaling operation on
multivector fields given by
\[
S^{(t)} \gamma = S^{(t)} \left( \sum_{j\geq 0} u^j \gamma_j \right) = \sum_{j\geq 0} (tu)^j \gamma_j\, .
\]
Let further
\[
\iota^{(t)}_\gamma = \iota_{S^{(t)}\gamma}\, .
\]
The family of dgla actions is then given by
\[
L_\gamma^{(t)}
= (1/u)(\co{u{\rm d}}{\iota^{(t)}_\gamma} + \iota^{(t)}_{u \dv_\Omega \gamma})
= \sum_{j\geq 0} (ut)^j (L_{\gamma_j} + t \iota_{\dv_\Omega \gamma_j})
\]
where $\gamma = \sum_{j\geq 0} u^j \gamma_j \in  T^\bullet[[u]]$.
\begin{prop}
\label{prop:Ltqiso}
For any $t\in \C$, $L_\gamma^{(t)}$ defines a
dgla module structure on $\Omega^\bullet[[u]]$.
Furthermore all these module structures are
$L_\infty$-isomorphic to each other.
\end{prop}
\begin{proof}
To show that the $L_\gamma^{(t)}$ are indeed dgla actions, compute
\begin{align*}
\co{ u{\rm d} }{ L_\gamma^{(t)} }
&=  \sum_{j\geq 0} (ut)^j t \co{ u{\rm d} }{ \iota_{\dv_\Omega \gamma_j} } \\
&= \sum_{j\geq 0} (ut)^{j+1} L_{\dv_\Omega \gamma_j} = L_{ u\dv_\Omega \gamma^{(t)} }\, .
\end{align*}
Furthermore
\begin{align*}
\co{ L_\gamma^{(t)} }{ L_\nu^{(t)} }
&= \sum_{j,k\geq 0} (ut)^{j+k} \co{L_{\gamma_j} + t \iota_{\dv_\Omega \gamma_j}}{L_{\nu_k} + t \iota_{\dv_\Omega \nu_k}} \\
&= \sum_{j,k\geq 0} (ut)^{j+k} \left(
L_{\co{\gamma_j}{\nu_k}} + t (-1)^{|\gamma_j|} \iota_{ \co{\gamma_j}{\dv_\Omega \nu_k} } +t\iota_{ \co{\dv_\Omega \gamma_j}{\nu_k} } \right) \\
&= \sum_{j,k\geq 0} (ut)^{j+k} \left(
L_{\co{\gamma_j}{\nu_k}} + t \iota_{\dv_\Omega \co{\gamma_j}{\nu_k} }\right) \\
&= L_{\co{\gamma}{\nu}}^{(t)}\, .
\end{align*}

Next we construct $L_\infty$ morphisms $H^{(t)}$
by integrating infinitesimal morphisms $h^{(t)}$.
Here the $h^{(t)}$ have only a single non-vanishing
Taylor coefficient of degree one, which we
denote (admittedly slightly confusing) by
\[
h^{(t)}_1(\gamma;\alpha) = -(-1)^{|\gamma|} h^{(t)}_\gamma \alpha\, .
\]
One finds that the (infinitesimal) $L_\infty$-morphism property
is equivalent to the following two conditions
for $h^{(t)}_\gamma$.
\begin{align*}
 -\frac{d}{dt}L^{(t)}_\gamma &= \co{u \rm d}{h^{(t)}_\gamma} + h^{(t)}_{u \dv_\Omega \gamma} \\
 h^{(t)}_{\co{\gamma}{\nu}_S}
&= \co{h^{(t)}_\gamma}{L^{(t)}_\nu} +(-1)^{|\gamma|}\co{L^{(t)}_\gamma}{h^{(t)}_\gamma}
\end{align*}

We claim that that
\[
 h^{(t)}_\gamma = -\frac{1}{u}\frac{d}{dt}\iota^{(t)}_\gamma
\]
satisfies these equations.\footnote{Note that
the expression on the right is well defined
since $\frac{d}{dt}\iota^{(t)}_\gamma \sim O(u)$.}
Compute
\begin{align*}
 \frac{d}{dt}L_\gamma^{(t)} &=
 (1/u)\co{u{\rm d}}{\frac{d}{dt}\iota^{(t)}_\gamma }
 + (1/u)\frac{d}{dt}\iota^{(t)}_{u \dv_\Omega \gamma} \\
&= - \co{u{\rm d}}{h^{(t)}_\gamma}
-h^{(t)}_{u\dv_\Omega \gamma}\, .
\end{align*}
In second order
\begin{align*}
\co{h^{(t)}_\gamma}{L^{(t)}_\nu} + (-1)^{|\gamma|} \co{L^{(t)}_\gamma}{h^{(t)}_\nu}
&=
-\sum_{j,k}(tu)^{j+k}\left( (j/t)\iota_{\co{\gamma_j}{\nu_k}}+(k/t)\iota_{\co{\gamma_j}{\nu_k}} \right)\\
&=
-\frac{d}{dt}\sum_{j,k}(tu)^{j+k}\iota_{\co{\gamma_j}{\nu_k}}\\
&=h^{(t)}_{\co{\gamma}{\nu}}  \, .
\end{align*}
\end{proof}

In the special case $t=0$ the action becomes
\[
L_\gamma^{(0)} \alpha = L_{\gamma_0}\alpha
\]
and in the case $t=1$
\[
L_\gamma^{(1)} \alpha = L_\gamma \alpha + \iota_{\dv_\Omega \gamma}\alpha.
\]
The quasi\ndash isomorphism between these two
structures is given by
\[
H^{(1)} = e^{\int_0^1dt h^{(t)}} = e^{\pm \iota^+ / u}
\]
where $\iota^+_\gamma = \iota^{(1)}_\gamma - \iota^{(0)}_\gamma = \sum_{j\geq 1} u^j \iota_{\gamma_j}$.

\subsection{Multivector field valued top forms
$VT^\bullet(M)$ (module)}
We define the multivector field valued top forms
\[
 VT^\bullet(M) := \Omega^d(M;\wedge^\bullet TM)\, .
\]
There is a natural non-degenerate pairing
\begin{gather*}
\skp{\cdot}{\cdot}:VT^\bullet(M)\otimes \Omega^{d-\bullet}_c(M)\to \C \\
\skp{\nu \Omega}{\alpha} = \int_M \Omega (\iota_\nu \alpha)\, .
\end{gather*}
Its obvious $u$-bilinear extension allows
for dualizing the dgla-module structures $L$ and $L^{(t)}$
on $\Omega^\bullet(M)[[u]]$ discussed above to dgla-module
structures on $VT^\bullet(M)[[u]]$.
We denote these dual module structures also by $L^{(t)}$ and hope that no
confusion arises.
Concretely, in our sign conventions the differential, temporarily called $\delta$,
and action are defined such that
\begin{align*}
 \skp{\delta(\nu \Omega)}{\alpha} &= -(-1)^{|\nu|}\skp{\nu \Omega}{u{\rm d}\alpha} \\
 \skp{L^{(t)}_\gamma(\nu \Omega)}{\alpha} &= -(-1)^{|\nu||\gamma|}\skp{\nu \Omega}{L^{(t)}_\gamma\alpha}\, .
\end{align*}

\begin{lemma}
The dgla module structure $L^{(t)}$ on $VT^\bullet(M)[[u]]$ is given
explicitly by the following data: The differential is $\delta=u\dv$ with
\[
 \dv (\nu \Omega) := (\dv_\Omega \nu)\Omega.
\]
The action is
\[
 L^{(t)}_\gamma (\nu \Omega) = \sum_{j\geq 0} (tu)^j
\left( \co{\gamma}{\nu}_S\Omega +(-1)^{|\gamma_j|}(1-t)(\dv_\Omega\gamma \wedge \nu)\Omega \right)
\]
where $\gamma=\sum_{j\geq 0} u^j \gamma_j$.
\end{lemma}
\begin{proof}
 Note first that
\[
 \int_M (\iota_\gamma \alpha)\Omega=\int_M \alpha\wedge \iota_\gamma \Omega\, .
\]
It follows that
\begin{align*}
 \skp{\delta(\nu \Omega)}{\alpha}
&=
-(-1)^{|\nu|}u\int_M (\iota_\nu{\rm d}\alpha) \Omega
& &=
-(-1)^{|\nu|}u\int_M ({\rm d}\alpha) \iota_\nu \Omega \\
&=
(-1)^{|\nu|+|\alpha|}u\int_M \alpha {\rm d}\iota_\nu \Omega
& &=
u\int_M \alpha \iota_{\dv_\Omega \nu} \Omega \\
&=
u\int_M (\iota_{\dv_\Omega \nu} \alpha) \Omega
& &=
 \skp{u \dv(\nu \Omega)}{\alpha}\, .
\end{align*}
In the fourth line we used that everything is zero unless $|\alpha|=|\gamma|$.
Furthermore, note that by a small computation
\[
 \int_M (L_\gamma \alpha) \Omega = -\int_M (\iota_{\dv_\Omega \gamma}\alpha )\Omega \, .
\]
Hence we obtain
\begin{align*}
 \skp{L^{(t)}_{u^j \gamma_j}(\nu \Omega)}{\alpha}
&=
-(-1)^{|\nu||\gamma_j|}
u^j\int_M\iota_\nu \left(t^j L_{\gamma_j}\alpha+t^{j+1}\iota_{\dv_\Omega \gamma_j} \alpha\right) \Omega \\
&=
-(-1)^{|\nu||\gamma_j|}
(tu)^j
\int_M
\left(\iota_{\co{\nu}{\gamma_j}_S}
+(-1)^{(|\nu|+1)|\gamma_j|} L_{\gamma_j}\iota_\nu \alpha
\right. \\ &\qquad\qquad\qquad\qquad\qquad\qquad \left.
+(-1)^{(|\nu|+1)|\gamma_j|}t\iota_{\dv_\Omega \gamma_j\wedge \nu} \alpha\right) \Omega \\
&=
-(-1)^{|\nu||\gamma_j|}
(tu)^j
\int_M
\left(\iota_{\co{\nu}{\gamma_j}_S}
+(-1)^{(|\nu|+1)|\gamma_j|} \iota_{\dv_\Omega \gamma_j\wedge \nu} \alpha
\right. \\ &\qquad\qquad\qquad\qquad\qquad\qquad \left.
+(-1)^{(|\nu|+1)|\gamma_j|}t\iota_{\dv_\Omega \gamma_j\wedge \nu} \alpha\right) \Omega \\
&=
-(-1)^{|\nu||\gamma_j|}(tu)^j
\int_M
\left(
-(-1){|\nu||\gamma_j|}\iota_{\co{\gamma_j}{\nu}_S}
-(-1)^{(|\nu|+1)|\gamma_j|} \iota_{\dv_\Omega \gamma_j\wedge \nu} \alpha
\right. \\ &\qquad\qquad\qquad\qquad\qquad\qquad \left.
+(-1)^{(|\nu|+1)|\gamma_j|}t\iota_{\dv_\Omega \gamma_j\wedge \nu} \alpha\right) \Omega \\
&=
\skp{(tu)^j(\co{\gamma}{\nu}_S+(-1)^{|\gamma_j|}(1-t) \dv_\Omega \gamma_j\wedge \nu)\Omega}{\alpha} \, .
\end{align*}

\end{proof}

In view of the PSM morphism, the most interesting
case is $t=1$. Here the action is the pushforward
of the adjoint action along the isomorphism
\begin{align*}
 T^\bullet(M)[[u]] &\to VT^\bullet(M)[[u]] \\
\gamma &\mapsto \gamma \otimes \Omega \, .
\end{align*}

\subsection{The Hochschild chains (module)}
The (normalized) Hochschild chain complex of the
algebra $A$ is the complex
\[
C_{-\bullet}(A) = A\otimes \bar{A}^{\otimes \bullet}
\]
where $\bar{A}=A/\C\cdot 1$.
It is equipped with differential $b_H$
\[
b_H(a_0\otimes \cdots \otimes a_n) = a_0a_1\otimes a_2\otimes \cdots \otimes a_n \pm \dots +(-1)^n a_na_0\otimes a_1\otimes \cdots \otimes a_{n-1}.
\]
The normalized Hochschild cochain complex acts on
the normalized chain complex through the (dgla) action
\begin{multline}
\label{equ:hochaction}
L_D (a_0\otimes \cdots \otimes a_n) =
\sum_{j=n-d+1}^{n}(-1)^{n(j+1)}  D(a_{j+1}, \dots, a_0, \dots) \otimes a_{d+j-n}\otimes \cdots \otimes a_j
+\\
+\sum_{i=0}^{n-d}(-1)^{(d-1)(i+1)}a_0 \otimes \cdots \otimes a_i \otimes D(a_{i+1},\dots, a_{i+d}) \otimes \cdots \otimes a_n.
\end{multline}
In particular $b_H=L_{m_0}$.

\subsection{The cyclic chains (module)}
The normalized Hochschild chain complex is equipped
with an additional differential $B$ of degree -1
discovered by Rinehart and rediscovered by Connes.
\[
 B(a_0\otimes \cdots\otimes a_n)
=
\sum_{j=0}^n (-1)^{jn} 1\otimes a_j\otimes \cdots \otimes a_n\otimes a_0\otimes \cdots \otimes a_{j-1}
\]
One can check that this differential (graded) commutes
with the action \eref{equ:hochaction} above, and
hence anticommutes with $b_H$. Introducing an
additional formal variable $u$ of degree +2, one
defines the negative cyclic chain complex as
\[
(C_{\bullet}(A)[[u]], b_H+uB).
\]
Its homology is called the negative cyclic homology.
Other cyclic homology theories can be obtained
from the negative cyclic complex by tensoring with
an appropriate $\C[u]$-module and will not receive
specialized treatment in this paper.

\subsection{Hochschild complex -- sheaf version $E^\bullet$ (module)}
Consider the sheaf $D^\bullet(M)$ of $\bullet$-differential
operators. E.g., $D^1(M)$ is the sheaf of differential
operators. It is a complex with the Hochschild
differential\footnote{Note that this is not the $b_H$
from above, there is no $a_0 \Phi(a_1,a_2,\dots,a_n)$-term. }
\[
(b \Phi)(a_0,\dots,a_n) = \Phi(a_0 a_1,a_2,\dots,a_n) \pm \Phi(a_0, a_1,\dots,a_{n-1}a_n) \pm \Phi(a_n a_0,a_1,\dots,a_{n-1})
\]
Also, note that there is an action of the cyclic
group(oid) on $D^\bullet(M)$ generated by
\[
(\sigma \Phi)(a_0,\dots,a_n) = (-1)^n \Phi(a_1,a_2,\dots,a_n, a_0).
\]

There is a canonical flat connection $\nabla$ on
$D^\bullet(M)$, compatible with the differential
and the cyclic action. It is given by the de
Rham differential:
\[
(\nabla \Phi)(a_0,\dots,a_n) = {\rm d} (\Phi(a_0,\dots,a_n)).
\]

\begin{defi}
The extended Hochschild cochain complex is the
total complex
\[
 E^\bullet=(\Gamma(D^\bullet(M)\otimes_{C^\infty(M)} \Omega^\bullet(M)), b + \nabla)\, .
\]
The normalized extended Hochschild complex
$E^\bullet_{norm}$ is the subcomplex of
multidifferential operators $\Phi$ such that
\[
\Phi(a_0,\dots ,a_{j-1},1,a_{j+1},\dots,a_n) = 0
\]
for all $a_0,\dots a_n$ and all $j=1,..,n$.
\end{defi}

This complex $E^\bullet$ is just another complex
computing Hochschild cohomology with values in
$\Omega^d(M)$, as the following proposition shows.

\begin{prop}
\label{prop:Eqiso}
The embedding $C^\bullet(A,\Omega^d(M)) \to E^\bullet$ given by
\[
\Phi \mapsto ((a_0,\dots, a_n)\mapsto a_0\Phi(a_1,...,a_n) )
\]
is a quasi\ndash isomorphism. 
\end{prop}

We will benefit from the following elementary result.
\begin{lemma}
\label{lem:cohomhelper}
 Let $(K^{p,q})_{0\leq p\leq n, q\in \Z}$ be a
double complex with differential $d_1+d_2$, where
\begin{align*}
 d_1:K^{p,q}\to K^{p+1,q} &\quad\quad d_2:K^{p,q}\to K^{p,q+1}.
\end{align*}
Then the following holds:
\begin{enumerate}
 \item If the $d_1$-cohomology is concentrated in
bottom degree $p=0$, then the inclusion of the
$d_1$-closed, $p$-degree 0 elements
\[
 \{k\in K^{0,\bullet}\mid d_1k=0 \} \hookrightarrow K^{\bullet,\bullet}
\]
is a quasi\ndash isomorphism.
 \item If the $d_1$-cohomology is concentrated in
top degree $p=n$, then the projection onto the top
$p$ degree elements modulo exact elements
\[
 K^{\bullet,\bullet} \twoheadrightarrow K^{n,\bullet}/d_1 K^{n-1,\bullet}
\]
is a quasi\ndash isomorphism.
\end{enumerate}
\end{lemma}
\begin{proof}
 At least the first statement is probably familiar
to the reader. The proof of the second statement is
essentially dual to the proof of the first.
\end{proof}

\begin{proof}[Proof of Proposition \ref{prop:Eqiso}]
It is more or less obvious that the above map is a
map of complexes. It remains to be shown that it
is a quasi\ndash isomorphism.

Let us compute the cohomology of $E^\bullet$ wrt.
$\nabla$, i.e., the first term in the spectral
sequence associated $E^\bullet$. We claim that it
is concentrated in the top form-degree $d=\dim M$,
and every class has exactly one representative in
the image of the above quasi\ndash isomorphism. To show
this, consider the spectral sequence associated to
the filtration on multidifferential operators by
the degree in the first ``slot'' (i.e., the slot
in which $a_0$ is inserted). The first term in this
spectral series is the associated graded, i.e.,
multidifferential operators with values in
$\wedge^\bullet T^*M \otimes S^\bullet TM$.
The differential $d_0$ is, in local coordinates,
the operator $d_0 = \sum_i(dx_i\wedge )\otimes (\p_i \cdot)$,
multiplying the $\wedge^\bullet T^*M$-part by
$dx^i$ and the $S^\bullet TM$-part by $\p_i$.
The cohomology is concentrated in form degree $d$
and operator degree $0$. Probably the quickest way
to see this is to note that the complex
$\wedge^\bullet T^*M \otimes S^\bullet TM$ with
the above differential is isomorphic to the
Koszul complex of $S^\bullet TM$, the isomorphism
being given by contracting the first factor
with a section of $\wedge^d TM$. The spectral
sequence degenerates at this point by (form\ndash)degree
reasons. This means that any $\nabla$-cohomology
class has exactly one representative of form degree
$d$ and of differential operator degree $0$ in the first slot.
 This proves the above claim, and hence the proposition.
\end{proof}

\subsection{Cyclic Cochains -- sheaf version (module)}

\begin{defi}
The extended cyclic complex is the complex
$(E^\bullet)^\sigma$ of invariants under the
cyclic action. The extended cyclic $(b,B)$-complex
is the complex $E^\bullet_{norm}[[u]]$ with differential
$b+uB$, where $B$ is Connes' $B$.
\end{defi}

For an orientable manifold, this complex computes the cyclic cohomology.

\begin{prop}
For $M$ orientable, the cohomology of the extended cyclic complexes $(E^\bullet)^\sigma$ and $E^\bullet_{norm}[[u]]$ is the cyclic cohomology of $C^\infty(M)$.
\end{prop}
\begin{proof}
Consider again the spectral sequence and compute the
$\nabla$-cohomology of the two complexes. As in the
last proof, the first term of the spectral sequence
for $E^\bullet_{norm}[[u]]$ is, as a vector space,
isomorphic to $D^\bullet_{norm}[[u]]$, the isomorphism
being given in the last proposition. One can see more
or less by the definitions that the differentials $b,B$
are mapped to $b_H,B$ under this isomorphism.

For the case of $(E^\bullet)^\sigma$, note that $\nabla$
commutes with the action of the cyclic group. It follows
that taking the $\nabla$-cohomology commutes with taking
cyclic invariants. The result then follows as in the
proof of the last proposition.
\end{proof}

\section{Part II: The meaning of the PSM morphism}
\label{sec:partII}
\subsection{The original PSM morphism}
Let $M$ be orientable and choose a volume form $\Omega$.
The original PSM morphism $\mV_{PSM,orig}$
is an $L_\infty$-morphism of modules over
$(T^\bullet(M)[[u]], u \dv_\Omega, \co{\cdot}{\cdot}_S)$,
constructed by the first two authors in \cite{cafe}
using essentially an equivariant version of the
Poisson sigma model. The two modules it relates
are the cyclic chains and the multivector fields.
\[
 \mV_{PSM,orig}: (C_\bullet(A,A)[[u]],b+uB) \to (T^\bullet(M)[[u]], u \dv_\Omega).
\]
The module structure on the left is given by pulling
back the $C^\bullet(A)$-action along $\mU_K$.
The module structure on the right is the trivial
module structure (!). We copy the following proposition
from \cite{cafe}

\begin{prop}
\label{prop:CF}
The morphism $\mV_{PSM,orig}$ is a morphism of
$L_\infty$-modules (but not a quasi\ndash isomorphism).
\end{prop}

\subsection{The (reinterpreted) PSM morphism $\mV^*_{PSM}$}
Here we give a new interpretation of the above morphism
The (reinterpreted) PSM morphism $\mV^*_{PSM}$
is a quasi\ndash isomorphism of $L_\infty$-modules over the dgla
$(T^\bullet(M)[[u]], u \dv_\Omega, \co{\cdot}{\cdot}_S)$. However,
the two modules are the multivector\ndash field\ndash valued top forms,
which can be identified with $T^\bullet(M)[[u]]$ using the
volume form, and the extended cyclic complex $E_{norm}^\bullet[[u]]$.
\[
 \mV^*_{PSM} : (T^\bullet(M)[u],u\dv) \cong (VT^\bullet(M)[u], ud) \to (E_{norm}^\bullet[u], \nabla+b+uB).
\]
The dgla module structure on the very left is the adjoint
one, in contrast to the trivial one above, and on the
middle $L^{(1)}$. The $L_\infty$-module
structure on the right is defined via pullback of the
dgla action of $C^\bullet(A)$ via the (Kontsevich)
$L_\infty$-morphism $\mU^{(0)}$.

The reinterpreted morphism is constructed from the
original one as follows:
\[
 \mV^*_{PSM}(\gamma_1,..,\gamma_m)(\gamma)(a_0,..,a_n)
=
\iota_{\mV_{PSM,orig}(\gamma_1,..,\gamma_m,u\gamma;a_0,..,a_n)}\Omega.
\]

\begin{thm}
 The morphism $\mV^*_{PSM}$ is a quasi\ndash isomorphism of $L_\infty$-modules.
\end{thm}
\begin{proof}
 The fact that it is an $L_\infty$\ndash morphism is an easy
consequence of Proposition \ref{prop:CF} and the previous
observation that for any multivector field $\nu$
\[
 \iota_{\dv \nu} \Omega = d \iota_\nu \Omega.
\]

It remains to be shown that the zero-th Taylor component
is an isomorphism on cohomology. In view of Lemma \ref{lem:cohomhelper}
it is sufficient to show that the composition with the
projection onto the top form degree part modulo the image of
$\nabla$ is a quasi\ndash isomorphism. Explicit computation yields
that the 0-th Taylor component is
\[
 \gamma \mapsto \pm ( (a_0,..,a_k) \mapsto a_0\gamma(a_1,..,a_k)\Omega)+ (\text{lower form degree}).
\]
The first part is the HKR morphism, known to be a quasi\ndash isomorphism,
and the remainder does not matter due to the projection onto top
form degree components.
\end{proof}

\appendix

\section{Our signs conventions}
There are many signs involved in the discussions above. Since sign computations are typically lengthy and boring, we did not explain them all. However, we list here the underlying conventions for the reader who believes $1\neq -1$ and wants to check.

Let $\alg{g}^\bullet$ be a graded vector space. An $L_\infty$-algebra structure on $\alg{g}^\bullet$ is a degree 1 coderivation $Q$ on the cofree (graded) cocommutative coalgebra without counit cogenerated by $\alg{g}^{\bullet+1}$, i.e. $S^+\alg{g}^{\bullet+1}$, satisfying $Q^2=0$. Any such coderivation is determined by its Taylor coefficients
\[
Q_n(x_1,\dots,x_n) = \pi Q(x_1,\dots,x_n)
\]
where $\pi$ is the projection on $\alg{g}^{\bullet+1}\subset S^+\alg{g}^{\bullet+1}$. If $\alg{g}$ carries the structure $(d,\co{\cdot}{\cdot})$ of a dgla, we associate to it an $L_\infty$-structure by the following convention (others are possible)
\begin{align*}
Q_1(x) &= dx & Q_2(x_1,x_2)=-(-1)^{|x_1|}\co{x_1}{x_2}\, .
\end{align*}

An $L_\infty$-module structure on the graded vector space $M^\bullet$ is a coderivation $\tilde{Q}$ lifting $Q$ on the cofree comodule $S\alg{g}^{\bullet+1}\otimes M^\bullet$. Again, it is determined by its Taylor coefficients $\pi_M\circ \tilde{Q}$. We identify (by convention) a dgla module $(M^\bullet,\delta,L)$ over the dgla $\alg{g}$ with the $L_\infty$\ndash module
\begin{align*}
 \tilde{Q}_0(m) &= \delta m & \tilde{Q}_1(x;m)=-(-1)^{|x_1|}L_x m\, .
\end{align*}

Next let $\hat{M}^\bullet$ be another graded vector space and $\skp{\cdot}{\cdot}$ be a nondegenerate pairing between $\hat{M}^\bullet$ and $M^\bullet$. This allows us to endow $\hat{M}^\bullet$ with an $L_\infty$\ndash structure $\tilde{Q}^*$ defined by
\[
\skp{\tilde{Q}^*_n(x_1,\dots,x_n;\hat{m})}{m}=-(-1)^{|\hat{m}|(n+1+\sum_j|x_j|)}\skp{\hat{m}}{\tilde{Q}_n(x_1,\dots,x_n;m)}\, .
\]

Let $M^\bullet,N^\bullet$ be $L_\infty$\ndash modules. A morphism $\phi$ between them is a degree zero morphism of the comodules intertwining the coderivations. It is also determined by the Taylor coefficients $\pi_N\phi$.
Let $\hat{N}^\bullet,\hat{M}^\bullet$ be $L_\infty$\ndash modules, with the module structure determined by nondegenerate pairings as above. Then one can define an adjoint morphism $\phi^\star$ from $\hat{N}$ to $\hat{M}$ by the formula
\[
\skp{\phi_n^*(x_1,\dots,x_n;\hat{n})}{m} = (-1)^{|\hat{m}|(n+\sum_j|x_j|)} \skp{\hat{n}}{\phi_n(x_1,\dots,x_n;m)}\, .
\]

Finally, let us describe the signs involved in section \ref{sec:partII}. Let $Q$ be the coderivation determining the $L_\infty$\ndash algebra structure on $T^\bullet(M)[[u]]$. Then the (adjoint) $L_\infty$\ndash module structure on  $T^\bullet(M)[[u]]$ is simply given by
\[
 \tilde{Q}_n(x_1,\dots,x_n;x)=Q_{n+1}(x_1,\dots,x_n;x).
\]
Let $\tilde{P}$ determine the $L_\infty$ module structure on $C_\bullet(A,A)[[u]]$. Then the module structure on $E_{norm}^\bullet[[u]]$ is determined by the coderivation $\tilde{O}$, defined such that for a map $\lambda: C_\bullet(A,A)[[u]]\to T^\bullet(M)[[u]]$:
\[
 \tilde{O}_n(x_1,\dots,x_n;\iota_{\lambda(\cdot)} \Omega)
=
-(-1)^{|\lambda|(n+1+\sum_j|x_j|)}
\iota_{\lambda(\tilde{P}_n(x_1,\dots,x_n;\cdot))} \Omega
+\delta_{n,0} \nabla \iota_{\lambda(\cdot)} \Omega\, .
\]

\nocite{*}
\bibliographystyle{plain}
\bibliography{cfmorph}

\end{document}